\documentclass[12pt, a4paper,final]{article} 			
\pdfoutput=1
\usepackage[english]{babel}
\usepackage[utf8]{inputenc}  							
\usepackage{amsmath}
\usepackage{amssymb}
\usepackage{amsthm}
\usepackage{color}
\usepackage{enumerate}									
\usepackage{xspace}										
\usepackage{stix}										
\usepackage{caption}
\usepackage{authblk}									
\usepackage{float}

\usepackage{standalone}									
\usepackage{xcolor}										
\definecolor{c1}{named}{red}
\definecolor{c2}{named}{blue}
\definecolor{c3}{named}{brown}
\definecolor{c4}{named}{cyan}
\definecolor{c5}{named}{green}
\definecolor{c6}{named}{magenta}
\definecolor{c7}{named}{orange}
\definecolor{c8}{named}{violet}
\usepackage{tikz}										
\usetikzlibrary{arrows}
\usetikzlibrary{shapes.geometric}
\usepackage{pgfplots}

\newenvironment{customlegend}[1][]{%
	\begingroup
	\csname pgfplots@init@cleared@structures\endcsname
	\pgfplotsset{#1}%
}{%
	\csname pgfplots@createlegend\endcsname
	\endgroup
}%

\def\addlegendimage{\csname pgfplots@addlegendimage\endcsname}

\usepackage{ragged2e}									
\usepackage[style = numeric,sorting = nyt,sortcites,firstinits = true,isbn = false,doi = false, url = false, clearlang = true,backend = biber]{biblatex}
\AtEveryBibitem{\clearfield{pagetotal}}					
\AtEveryBibitem{\clearfield{note}}						
\usepackage{csquotes}	
\DeclareRedundantLanguages{English,english,german,french, eng}{english,german,ngerman,french} 

\renewbibmacro{in:}{%
	\ifentrytype{article}{}{\printtext{\bibstring{in}\intitlepunct}}}		

\usepackage{bbm}										

\usepackage{geometry}
\allowdisplaybreaks[1]		

\usepackage{aliascnt}
\usepackage[pdfpagelabels=true,plainpages=false, hidelinks]{hyperref}

\begin{document}

\newtheorem{theorem}{Theorem}[section]
\newtheorem{theorema}{Theorem}						
\renewcommand{\thetheorema}{\Alph{theorema}}
\providecommand*{\theoremaautorefname}{Theorem} 
\newtheorem*{theorem*}{Theorem}

\newtheorem*{maintheorem}{Main Theorem}

\newtheorem*{conj}{Conjecture}

\newaliascnt{lemma}{theorem}
\newtheorem{lemma}[lemma]{Lemma}
\aliascntresetthe{lemma}
\providecommand*{\lemmaautorefname}{Lemma} 

\newaliascnt{cor}{theorem}
\newtheorem{cor}[cor]{Corollary}
\aliascntresetthe{cor}
\providecommand*{\corautorefname}{Corollary} 

\newtheorem*{rem}{Remark}

\newtheorem*{defi}{Definition}

\newaliascnt{prop}{theorem}
\newtheorem{prop}[prop]{Proposition}
\aliascntresetthe{prop}
\providecommand*{\propautorefname}{Proposition} 



\newcommand{\enmath}[1]{\ensuremath{#1}\xspace} 
\newcommand{\defeq}{\stackrel{\mathrm{def}}{=}}
\newcommand{\RR}{\mathbb{R}}
\newcommand{\ZZ}{\mathbb{Z}}
\newcommand{\NN}{\mathbb{N}}
\newcommand{\CC}{\mathbb{C}}
\newcommand{\HH}{\mathbb{H}}
\newcommand{\KK}{\mathbb{K}}
\newcommand{\OO}{\mathcal{O}}
\newcommand{\Q}{\mathcal{Q}}
\newcommand{\fix}[1]{\operatorname{Fix}\left(#1\right)}
\newcommand{\tr}[1]{\operatorname{tr}\left(#1\right)}
\newcommand{\rk}{\operatorname{rank}}
\newcommand{\codim}{\operatorname{codim}}
\newcommand{\FF}{\mathcal F}
\newcommand{\GG}{\mathcal G}
\newcommand{\Hh}{\mathcal H}
\newcommand{\ROT}[1]{\textbf{\textcolor{red}{(#1)}}}			
\newcommand{\Hom}[2]{\operatorname{Hom}_\Sigma (#1, #2)}
\newcommand{\End}[1]{\operatorname{End}_\Sigma (#1)}
\newcommand{\Endnil}[1]{\operatorname{End}_\Sigma^{\mathrm{nil}} (#1)}
\newcommand{\Endnilnull}[1]{\operatorname{End}_\Sigma^{\mathrm{b-nil}} (#1)}
\newcommand{\spec}[1]{\operatorname{spec(#1)}}
\newcommand{\Ind}[1]{\operatorname{ind}\left( #1 \right)}
\newcommand{\im}[1]{\operatorname{Im}\left( #1 \right)}
\newcommand{\Endquot}[1]{\End{#1} / \Endnil{#1}}
\newcommand{\Iso}{\operatorname{Iso}\left( W_1 \oplus \ldots \oplus W_k \right)}
\newcommand{\ind}[1]{\dim \Endquot{#1}}
\newcommand{\gker}[1]{\operatorname{ker}_0 (#1)}
\newcommand{\redim}[1]{\operatorname{im}_0 (#1)}
\newcommand{\M}[2]{\operatorname{M} (#1 ; #2 )}
\newcommand{\Mi}[2]{\operatorname{M}_i \left(#1 ; #2 \right)}
\newcommand{\Mr}[2]{\mathbf{M}_r \left(#1 ; #2 \right)}
\newcommand{\Mnil}[2]{\operatorname{M}_1^{\mathrm{nil}} \left(#1 ; #2 \right)}
\newcommand{\quadmat}[2]{\begin{pmatrix}
		#1 ^{11}	& \cdots	& #1 ^{1 #2} \\
		\vdots		&			& \vdots \\
		#1 ^{#2 1}	& \cdots	& #1 ^{#2 #2}
\end{pmatrix}}

\renewcommand*{\mkbibnamefamily}[1]{\textsc{#1}}
\renewcommand*{\finalnamedelim}{%
	\ifnumgreater{\value{liststop}}{2}{}{}%
	\addspace\&\space}
\renewcommand\autocite[1]{\textcite{#1}}

\renewcommand{\sectionautorefname}{Section}
\renewcommand{\subsectionautorefname}{Subsection}

\numberwithin{equation}{section} 
\renewcommand{\theequation}{\arabic{section}.\arabic{equation}} 


\title{Generic steady state bifurcations in monoid equivariant dynamics with applications in homogeneous coupled cell systems}
\author{Sören Schwenker\thanks{Department of Mathematics, Universität Hamburg, Germany, soeren.schwenker@uni-hamburg.de}}
\date{}
\maketitle

\begin{abstract}
	\noindent
	We prove that steady state bifurcations in finite-dimensional dynamical systems that are symmetric with respect to a monoid representation generically occur along an absolutely indecomposable subrepresentation. This is stated as a conjecture in B. Rink and J. Sanders, ``Coupled cell networks and their hidden symmetries'', SIAM J. Math. Anal., 46 (2014). It is a generalization of the well-known fact that generic steady state bifurcations in equivariant dynamical systems occur along an absolutely irreducible subrepresentation if the symmetries form a group -- finite or compact Lie. Our generalization also includes non-compact symmetry groups. The result has applications in bifurcation theory of homogeneous coupled cell networks as they can be embedded (under mild additional assumptions) into monoid equivariant systems.
\end{abstract}

\section*{Introduction}

In the last decades equivariant dynamics has gained a lot of attention. Symmetries arise frequently in nature and applications and can lead to numerous staggering phenomena such as pattern formation or synchronization of behavior. In equivariant dynamical systems, symmetries provide the underlying structure to explain such unexpected behavior. Examples are dynamically invariant subspaces, spectral degeneracies, complicated bifurcations, and many more. Research in this area has been very active and lots of remarkable results have been established. More details on equivariant dynamics can be found, for example, in \cite{Chossat.2000, Field.2007, Golubitsky.1985, Golubitsky.1988} with no claim of this list being complete. However, in most of these studies the symmetries in question need to have an underlying structure themselves. In particular they are required to form a group, most often a finite one or a compact Lie group.

In recent years the field of network dynamical systems has received increased activity and developments therein have called for less restrictive classes of symmetries. Network dynamical systems exhibit phenomena that resemble those that have been encountered in the context of equivariant dynamics such as multiple eigenvalues, high dimensional center subspaces, and similarly unexpected and complex bifurcation behavior. As a matter of fact, in this context we encounter dynamical systems that are equivariant with respect to linear symmetries as well. These, however, often do not form a group but less restrictive structures such as a groupoid (see \cite{Golubitsky.2006}), a semigroup or a monoid (see \cite{Rink.2014}).

In this article we investigate generic steady state bifurcations in one-parameter families of smooth vector fields that are equivariant with respect to a monoid. \textcite{Nijholt.2017c,Nijholt.2017d,Nijholt.2016,Rink.2013,Rink.2014,Rink.2015} show that under mild additional assumptions homogeneous coupled cell networks can be regarded as the restriction of equivariant systems to some invariant subspace. They prove numerous results on how to exploit these symmetries to investigate the dynamics of the network such as Lyapunov-Schmidt reduction \cite{Rink.2014}, normal forms \cite{Rink.2015}, center manifold reduction \cite{Nijholt.2017c} or by determining bifurcations in the extended system \cite{Nijholt.2017d}. These techniques provide an easy to apply step-by-step machinery to determine generic bifurcations in the network. However, one part could not be completely clarified: determining generic center subspaces of the equivariant system that will lead to bifurcations. In a general one-parameter family -- without any symmetry -- a generic steady state bifurcation occurs along a one-dimensional subspace as the kernel of the linearization is generically one-dimensional. In the context of group equivariant dynamics the picture is more complicated as symmetry may force kernels to be higher dimensional. However, it is well-known that generic steady state bifurcations occur along an absolutely irreducible subrepresentation. That is an invariant subspace that contains no further nontrivial invariant subspaces and whose endomorphism space is isomorphic to the real numbers.

A similar result so far only exists in a special case for one-parameter families of systems that are monoid equivariant. A monoid representation decomposes into indecomposable subspaces meaning they cannot be decomposed any further. These are once again invariant under the dynamics. Just as in the group case they can be of three types depending on their endomorphism space. \autocite{Rink.2014} show that whenever the representation decomposes into indecomposable subrepresentations that are pairwise nonisomorphic, steady state bifurcations in one-parameter families generically occur along an absolutely indecomposable subprepresentation. It is already anticipated in their paper that this result holds in full generality. However, no proof is given. In this article we close this gap by proving
\begin{maintheorem}
	Steady state bifurcations in one-parameter families of systems that are equivariant with respect to a finite-dimensional representation of a monoid generically occur along an absolutely indecomposable subrepresentation.
\end{maintheorem}
\noindent
This is the immediate generalization of the aforementioned result in the group context. It can be seen uncoupled from network dynamical systems.

At this point, we would like to mention an article by Nijholt and Rink that is currently available as a preprint (see \autocite{Nijholt.2017}). Therein the authors address the same question that is already mentioned in \cite{Rink.2014} from a much more general point of view. They investigate generic bifurcations in $k$-parameter families of monoid equivariant vector fields by determining which configuration of invariant subrepresentations generically occurs as generalized kernel or as center subspace. Their article also includes the results of this text as a special case. The proof, however, is a lot more involved as it is presented in an algebraic setting and makes use of algebraic geometry as well as noncommutative algebra -- most importantly in the form of Wedderburn's structure theorem. On the other hand, our result and the proof are presented from the perspective of the application to homogeneous coupled cell networks. For that reason it is not suitable to provide the result in even more generality. It may, however, be seen as an easy step into the theory of monoid representations coming from network dynamical systems.

We use \autoref{sec:rm} to wrap up the basics of representation theory of monoids that are needed throughout the text. The results are stated without proofs as they are nicely presented in \cite{Rink.2014}. \autoref{sec:gb} contains the proof of the main result. We start in \autoref{subsec:ic} by investigating nilpotent endomorphisms of representations that are direct sums of subrepresentations which are all pairwise isomorphic. These form building blocks of arbitrary representations and are called isotypic components. In \autoref{subsec:gb} we complete the proof by reducing the question of generic generalized kernels in an arbitrary representation to that of determining the nilpotent endomorphisms of its isotypic components. Finally, in \autoref{sec:ex} we investigate a homogeneous coupled cell system with eight cells as an example to illustrate the application of the main theorem. Some technical details on submanifolds of matrix manifolds are postponed to the appendix.

\section{Finite dimensional real monoid representations}
\label{sec:rm}

We use this chapter to summarize some results on finite-dimensional real representations of monoids. Note that we do not impose any further restrictions on the monoid throughout the whole article. It may be finite or infinite and, therefore, includes finite or compact symmetry groups. But also noncompact groups are included. The proofs for the results in this section are omitted. In \autocite{Rink.2014} they are stated for finite monoids. However, only the fact that the representations are finite-dimensional is used, and therefore they also apply in the case in consideration here.

\begin{defi}
	The tuple $(\Sigma, \cdot)$, where $\Sigma$ is a set and $\cdot \colon \Sigma \times \Sigma \to \Sigma$ is a map so that
	\begin{enumerate}[(i)]
		\item $ (\sigma \cdot \sigma') \cdot \tilde{\sigma} = \sigma \cdot ( \sigma' \cdot \tilde{\sigma} )$ for all $\sigma, \sigma', \tilde{\sigma} \in \Sigma$ \hfill (associativity)
		\item there exists an element $1 \in \Sigma$ so that $1 \cdot \sigma = \sigma \cdot 1 = \sigma$ for all $\sigma \in \Sigma $ \hfill (neutral element)
	\end{enumerate}
	is called a \emph{monoid}. We abbreviate $\sigma \sigma'= \sigma \cdot \sigma'$ and call $\Sigma$ the monoid if the multiplication is clear from context.
\end{defi}
The definition of a \emph{representation} is well known in the context of groups and we define it accordingly for monoids. Let $\Sigma$ be a monoid and let $V$ be a finite-dimensional real vector space and $\mathfrak{gl}(V)$ the space of linear maps from $V$ to itself. We call the monoid homomorphism
\[ A \colon \Sigma \to \mathfrak{gl}(V), \quad \sigma \mapsto A_\sigma, \]
with
\[ A_{\sigma \sigma'} = A_{\sigma} A_{\sigma'} \quad \text{for all} \quad \sigma, \sigma' \in \Sigma \]
and $A_1 = \mathbb{1}_V$ for the neutral element, a \emph{representation} of $\Sigma$. When the homomorphism is known from context or not needed in its explicit form we also call $V$ a representation of $\Sigma$. A subspace $W \subset V$ is called a \emph{subrepresentation} if it is invariant under the action of the monoid:
\[ A_{\sigma} W \subset W \quad \text{for all} \quad \sigma \in \Sigma. \]
The representation $V$ is called \emph{irreducible} if there exists no proper subrepresentation $W \subset V$ with $W \ne \lbrace 0 \rbrace, V$. It is called \emph{indecomposable} if there are no two proper subrepresentations $W, W'$ with $V= W \oplus W'$. Unlike for group representations, these two properties are not equivalent as an indecomposable representation need not be irreducible. However, the definition of indecomposability directly yields the existence of a decomposition of $V$ into indecomposable subspaces
\[ V = W_1 \oplus \dotso \oplus W_s. \]
This decomposition is unique up to equivalence of subrepresentations, which is stated in the following theorem.
\begin{theorem}[Krull-Schmidt]
	\label{thm:KrullSchmidt}
	Let $V$ be a representation of\/ $\Sigma$ and let 
	\[ V = W_1 \oplus \dotso \oplus W_s \]
	be a decomposition of $V$ into indecomposable subrepresentations. Then this decomposition is unique up to isomorphisms: If it also holds that
	\[ V = W_1' \oplus \dotso \oplus W_{s'}' \]
	with indecomposable subrepresentations $W_1', \dotsc , W_{s'}'$ then $s=s'$ and $W_i \cong W_i'$ for all $i$ after renumbering the subrepresentations.
\end{theorem}

Consider a second representation $A' \colon \Sigma \to \mathfrak{gl}(V')$ on a finite dimensional real vector space $V'$ and a linear map
\[ L \colon V \to V'. \]
If $L$ commutes with the monoid action
\[ L \circ A_\sigma = A'_\sigma \circ L \quad \text{for all} \quad \sigma \in \Sigma \]
we call it a \emph{homomorphism of representations} and write $L \in \Hom{V}{V'}$. If $V = V'$, we call $L$ an \emph{endomorphism} and write $\End{V} = \Hom{V}{V}$. If $L$ is invertible, we call it an \emph{isomorphism} and $V$ and $V'$ are \emph{equivalent} or \emph{isomorphic}. The following remark points out why endomorphisms are especially interesting for the study of monoid equivariant dynamics.
\begin{rem}
	\label{rem:Rink43}
	Assume that $F \colon V \to V$ is a continuously differentiable vector field and $x_0\in V$ is a point such that
	\begin{enumerate}[(i)]
		\item $x_0$ is an equilibrium point of $F$, i.e., $F(x_0)=0$ ;
		\item $x_0$ is $\Sigma$-symmetric, i.e. $A_\sigma x_0 = x_0$ for all $\sigma \in \Sigma$;
		\item $F$ is $\Sigma$-equivariant, i.e. $F \circ A_\sigma  = A_\sigma \circ F $ for all $\sigma \in \Sigma$.
	\end{enumerate}
	Then differentiation of $F \left( A_\sigma x \right) = A_\sigma F \left( x \right)$ at $x=x_0 = A_\sigma x_0$ yields
	\[ L_0 \circ A_\sigma = A_\sigma \circ L_0 \]
	with $L_0 = D_x F \left( x_0 \right)$ and hence $L_0 \in \End{V}$.
\end{rem}


When we consider an indecomposable representation the space of endomorphisms has some interesting properties itself.
\begin{prop}
	\label{prop:Rink44}
	Let $V$ be an indecomposable representation and let $L \in \End{V}$. Then $L$ is either invertible or nilpotent (i.e., there exists $n \in \NN$ such that $L^n=0$).
\end{prop}

\noindent
This result is also known as the \emph{Fitting lemma} and can be found, for example, in \autocite{Jacobson.1980}. As a corollary we obtain that the set of nilpotent endomorphisms of an indecomposable monoid representation
\[ \Endnil{V} = \left\lbrace L \in \End{V} \mid L \text{ is nilpotent} \right\rbrace \]
is an ideal in $\End{V}$. Factoring out this ideal we obtain the following lemma




\begin{lemma}[Schur's Lemma]
	\label{lem:Schur}
	Let $V$ be an indecomposable representation. The quotient 
	\[ \Endquot{V} \]
	is a real division algebra.
\end{lemma}

\begin{rem}
	This result, stemming from module theory, is a significant generalization of the original lemma of Schur on irreducible modules -- which is important in the representation theory of groups -- to indecomposable ones.
\end{rem}
\noindent
Recall that any finite-dimensional real associative division algebra is isomorphic to either $\RR, \CC$, or $\HH$. In the first case we say $V$ is a representation of \emph{real type} or an \emph{absolutely indecomposable} representation. In the other two cases it is called a representation of \emph{complex} or, respectively, of \emph{quaternionic type}. Furthermore, we define the \emph{index} of $V$ to be the dimension of the division algebra:
\[ \Ind{V} = \ind{V}. \]

The next result is an immediate consequence of the Fitting lemma. It investigates concatenations of homomorphisms of indecomposable representations. Even though it appears to be out of context here, it turns out to be useful later in the text.
\begin{prop}
	\label{prop:Rink42}
	Let $V$ and $V'$ be indecomposable representations, and consider two homomorphisms $L\in \Hom{V}{V'}$ and $K \in \Hom{V'}{V}$. If $K\circ L$ is invertible, both $L$ and $K$ are isomorphisms.
\end{prop}

The final result of this section discusses perturbations of endomorphisms. Any $L \in \End{V}$ imposes a decomposition of $V$ into the so-called \emph{generalized kernel} and \emph{reduced image} of $L$
\[ V = \gker{L} \oplus \redim{L}. \]
These are kernel and image of $L^n$ when $n = \dim V$ or equivalently the generalized eigenspace of the eigenvalue $0$ and the direct sum of the generalized eigenspaces of all nonzero eigenvalues. Furthermore, both are subrepresentations that are invariant under $L$. Note that this decomposition is the same used in \autocite{Rink.2014} even though it is defined differently.

\begin{lemma}
	\label{lem:Rink63}
	Let $L_0 \in \End{V}$ and denote the decomposition into generalized kernel and reduced image of $L_0$ by
	\[ V = \gker{L_0} \oplus \redim{L_0} \]
	with respect to which 
	\[ L_0 = \begin{pmatrix}
	L_0^{11}	& 0 \\
	0			& L_0^{22}
	\end{pmatrix} \]
	with $L_0^{11}$ nilpotent and $L_0^{22}$ invertible. Then there is an open neighborhood $U \subset \End{V}$ of the zero endomorphism $0 \in \End{V}$ and smooth maps 
	\[ \phi^{11} \colon U \to \End{\gker{L_0}} \quad \text{and} \quad \phi^{22} \colon U \to \End{\redim{L_0}}\]
	so that for every $L \in U$ 
	\[ L_0 + L \quad \text{is conjugate to} \quad 
	\begin{pmatrix}
	\phi^{11}(L)	& 0 \\
	0				& \phi^{22}(L)
	\end{pmatrix}. \]
	It holds that $\phi^{11}(L) = L_0^{11} + L^{11} + \mathcal{O}( \| L \| ^2)$ and $\phi^{22}(L) = L_0^{22} + L^{22} + \mathcal{O}( \| L \| ^2)$.
\end{lemma}

\section{Generic steady state bifurcations in real monoid representations}
\label{sec:gb}

In this section we aim to prove the main theorem. The principal part of the proof is a generalization of the proof in \autocite{Rink.2014}. Therefore, the structures of the proofs are similar and some of the notation is used again. The major difference is that we have to take special care of monoid representations containing multiple direct summands that are equivalent subrepresentations.

In order to prove the claim, let $\Sigma$ be a monoid which is represented on the finite-dimensional real vector space $V$ as before:
\[ \Sigma \to \mathfrak{gl}(V), \quad \sigma \mapsto A_\sigma . \]
For the rest of this text all representations are real and finite-dimensional. Furthermore, let
\[ F \colon V \times \RR \to V \]
be a smooth equivariant vector field that depends on a real parameter:
\[ F \left( A_\sigma x , \lambda \right) = A_\sigma F(x,\lambda) \quad \text{for all} \quad x \in V, \lambda \in \RR \text{ and } \sigma \in \Sigma. \]
We assume that for all $\lambda$ the vector field possesses a $\Sigma$-symmetric equilibrium $x_0$. The implicit function theorem implies that for varying parameter values, solution branches can only emerge from $x_0$ at a parameter value $\lambda_0$ when the linearization $L_{\lambda_0} = \mathrm{D}_x F(x_0, \lambda_0)$ is not invertible. Without loss of generality, we may assume $\lambda_0 = 0$. Furthermore, Lyapunov-Schmidt reduction tells us that possible new solution branches then locally occur along the generalized kernel (the generalized eigenspace to the eigenvalue $0$) of $L_0$. As we have seen in a remark in \hyperref[sec:rm]{Section \ref{sec:rm}}, the linearization is equivariant as well. Hence $L_\lambda = \mathrm{D}_x F(x_0, \lambda)$ is a one parameter family of endomorphisms of $V$. \autocite{Rink.2014} prove that the Lyapunov-Schmidt reduction can be performed to preserve equivariance. Hence, the bifurcation equation reduces to an equivariant equation on the generalized kernel. Thus, in order to prove the main theorem we have to investigate generalized kernels of one-parameter families of endomorphisms of finite-dimensional representations of monoids.

\subsection{Isotypic components and nilpotent endomorphisms}
\label{subsec:ic}
As a first step we decompose the representation $V$ into indecomposable components
\[ V = W_1 \oplus \dotso \oplus W_m \]
where each $W_i$ is an indecomposable subrepresentation of $V$. This decomposition is unique due to the Krull-Schmidt theorem (\ref{thm:KrullSchmidt}). It yields a partition of $\lbrace1, \dotsc , m\rbrace = P_1 \cup \dotso \cup P_s$ such that $W_i \cong W_j$ if $i$ and $j$ are in the same $P_k$ and $W_i \not\cong W_j$ if $i$ and $j$ are not in the same $P_k$. Summing up the components according to that partition
\[ V_k = \bigoplus_{i \in P_k} W_i, \]
i.e., summing up those components that are isomorphic or equivalent, we obtain a coarser decomposition
\[ V = V_1 \oplus \dotso \oplus V_s. \]
We call the $V_k$ \emph{isotypic components}. As it stems directly from the decomposition into indecomposable subrepresentations, this decomposition is unique up to equivalence of subrepresentations as well. As a matter of fact, we may identify each isotypic component with the finite direct sum of one of its indecomposable subrepresentations
\[ V_i \cong W_j^{s_i} \]
for some $s_i \in \NN$ and suitable $j\in \lbrace 1, \dotsc , m \rbrace$. We will prove some preparatory results on isotypic components first.

\begin{lemma}
	\label{lem:homisotypic}
	Let $X$ and $Y$ be indecomposable $\Sigma$-representations and let $V=X^s$ and $W=Y^r$ for some $r,s \in \NN$ be representations consisting of precisely one isotypic component. For $L \in \Hom{V}{W}$ and $K \in \Hom{W}{V}$ it holds that $KL \in \End{V}$, and we may represent it as a block matrix with respect to the decomposition:
	\[ KL = \quadmat{B}{s} \]
	so that $B^{ij} \in \End{X}$ for all $i,j$. Suppose that $X$ and $Y$ are nonisomorphic representations, i.e., $X \ncong Y$. Then all the $B^{ij}$ are nilpotent.
\end{lemma}
\begin{proof}
	As a first step we present $K$ and $L$ in block matrix form respecting the decompositions of $V$ and $W$, respectively:
	\[ L = \begin{pmatrix}
	L^{11}	& \cdots	& L^{1s} \\
	\vdots	&			& \vdots \\
	L^{r1}	& \cdots	& L^{rs}
	\end{pmatrix} \quad \text{and} \quad K = \begin{pmatrix}
	K^{11}	& \cdots	& K^{1r} \\
	\vdots	&			& \vdots \\
	K^{s1}	& \cdots	& K^{sr}
	\end{pmatrix} \]
	where $L^{ij} \in \Hom{X}{Y}$ and $K^{ij} \in \Hom{Y}{X}$ for all $i,j$. Therefore, the product $KL$ is an $s \times s$ block matrix with entries
	\[ (KL)^{ij} = \sum_{l=1}^{r} K^{il}L^{lj}, \]
	which are in $\End{X}$. In particular, this holds for each of the summands $K^{il}L^{lj} \in \End{X}$ for all $i,j,l$. \autoref{prop:Rink44} yields that such products are either invertible or nilpotent. Suppose now that $K^{il}L^{lj}$ is invertible for some $i,j,l$. Then \autoref{prop:Rink42} yields that both $K^{il}$ and $L^{lj}$ are isomorphisms and hence $X \cong Y$. This is a contradiction to our assumptions, and therefore $K^{il}L^{lj}$ are nilpotent for all $i,j,l$. The fact that $\Endnil{X}$ is an ideal (see \autoref{prop:Rink44}) provides the same result for finite sums of these elements and hence for all blockwise entries of the product $KL$ which completes the proof.
\end{proof}
\begin{rem}
	The previous lemma is a generalization of \autoref{prop:Rink42} on isotypic components.
\end{rem}

Next, we aim at understanding nilpotent endomorphisms of isotypic components in a similar fashion as for indecomposable representations. In order to do so we consider a real finite-dimensional representation consisting of precisely one isotypic component
\[ V = X^s \]
unless stated differently. Presenting the endomorphisms of $V$ in block matrix form as before we may identify 
\[ \End{V} \cong \M{s}{\End{X}}, \]
where $\M{s}{\End{X}}$ is the algebra of $s \times s$ matrices with entries in $\End{X}$. The first step in understanding the nilpotent endomorphisms $\Endnil{V}$ is to see that the collection $\Endnilnull{V}$ of matrices with blockwise nilpotent entries
\[ \Endnilnull{V} = \M{s}{\Endnil{X}} \]
is an ideal in $\End{V}$. This follows immediately from the fact that $\Endnil{X}$ is an ideal in $\End{X}$ (see \autoref{prop:Rink44}) and from the rules of matrix summation and multiplication.
\begin{lemma}
	\label{lem:ideal}
	The collection
	\[ \Endnilnull{V} \subset \End{V} \]
	is an ideal.
\end{lemma}
\begin{proof}
	Let
	\[ L = \quadmat{L}{s}, K= \quadmat{K}{s} \in \M{s}{\End{X}}. \]
	We may easily check the requirements for an ideal using the fact that $\Endnil{X}$ is an ideal.
	\begin{enumerate}[(i)]
		\item Obviously $0 \in \Endnilnull{V}$.
		\item Let $L,K \in \Endnilnull{V}$. Then $L^{ij},K^{ij} \in \Endnil{X}$ for all $i,j$, and therefore $-L^{ij},-K^{ij} \in \Endnil{X}$ as well as $L^{ij}-K^{ij} \in \Endnil{X}$ for all $i,j$. Thus $L-K \in \Endnilnull{V}$.
		\item Let $L \in \Endnilnull{V}$. Then
		\[ (LK)^{ij}= \sum_{l=1}^s L^{il}K^{lj} \]
		with each summand being an element in $\Endnil{X}$ as it is an ideal. Thus the same holds for the finite sum and $LK \in \Endnilnull{V}$. Exchanging the role of $L$ and $K$ yields the same for $KL$. \qedhere
	\end{enumerate}
\end{proof}
\begin{rem}
	Note that all elements in $\Endnilnull{V}$ are nilpotent themselves. This can be seen using the fact that after a choice of a basis for $V$ endomorphisms can be represented as real matrices with
	\[ \tr{L^k} = 0 \]
	for $L \in \Endnilnull{V}$ and all $k \in \NN$.
\end{rem}
As $\Endnilnull{V}$ is an ideal in $\End{V}$ or more precisely in $\M{s}{\End{X}}$, we may consider the factor ring/factor space
\[ \M{s}{\End{X}} / \M{s}{\Endnil{X}}. \]
This yields the decomposition
\begin{equation}
\label{eq:decomposition}
\M{s}{\End{X}} = \Endnilnull{V} \oplus W.
\end{equation}
where
\[ W \cong \M{s}{\End{X}} / \M{s}{\Endnil{X}}. \]
The isomorphism is exactly the projection map 
\[ \pi \colon \M{s}{\End{X}} \to \M{s}{\End{X}} / \M{s}{\Endnil{X}} \]
restricted to $W$. Note that factoring out $\Endnilnull{V}=\M{s}{\Endnil{X}}$ is the same as factoring out $\Endnil{X}$ entrywise. Therefore,
\[ \M{s}{\End{X}} / \M{s}{\Endnil{X}} = \M{s}{\Endquot{X}}. \]
The spaces are not only isomorphic but equal.
Remember from the considerations after Schur's lemma (\ref{lem:Schur}) that we may furthermore identify
\[ \Endquot{X} \cong \KK \]
where $\KK = \RR, \CC$, or $\HH$ depending on the representation type of $X$. Therefore, we may identify
\[ \M{s}{\Endquot{X}} \cong \M{s}{\KK} \]
using an isomorphism $\kappa$. It is important to bear in mind that, even though we identify the factor space with complex or even quaternionic matrices, we still treat it as a real algebra, meaning that we allow scalar multiplication by real numbers only.
\begin{lemma}
	\label{lem:nilpotentprojection}
	Let $L \in \End{V}$. Then $L$ is nilpotent if and only if $\pi(L)$ is nilpotent in $\M{s}{\Endquot{X}}$ and $\kappa \pi(L)$ is nilpotent in $\M{s}{\KK}$.
\end{lemma}
\begin{proof}
	The first direction of the proof follows directly from the fact that $\pi$ and $\kappa$ are homomorphisms of rings. Conversely, let $L \in \End{V}$ such that $\pi(L)$ is nilpotent in $\M{s}{\Endquot{X}}$ (or equivalently $\kappa\pi(L)$ nilpotent in $\M{s}{\KK}$). Then there exists $k \in \NN$ such that $\pi(L)^k = 0 \in \M{s}{\Endquot{X}}$. This is the same as $\pi(L^k) = 0 \in \M{s}{\Endquot{X}}$ or equivalently $L^k \in \Endnilnull{V}$. As mentioned in the last remark, the elements of $\Endnilnull{V}$ are nilpotent themselves, so $L^k$ and therefore $L$ are nilpotent.
\end{proof}
\noindent
Summarizing we have seen that
\begin{align*}
	\Endnil{V} 	&= \left\lbrace L \in \M{s}{\End{X}} \mid \pi (L) \text{ nilpotent} \right\rbrace \\
				&= \left\lbrace L \in \M{s}{\End{X}} \mid \kappa \pi (L) \text{ nilpotent} \right\rbrace.
\end{align*}

Recall from \autoref{eq:decomposition} that we may uniquely decompose elements $L \in \M{s}{\End{X}}$ as follows:
\begin{equation}
\label{eq:elt_decomp}
L = L_1 + L_2 \quad \text{where} \quad L_1 \in \Endnilnull{V}, L_2 \in W,
\end{equation}
which yields
\[ \pi (L) = \pi (L_2) \in \M{s}{\Endquot{X}}. \]
Therefore, using \autoref{lem:nilpotentprojection}, we obtain
\begin{align*}
L = L_1 + L_2 \in \Endnilnull{V} \oplus W \text{ nilpotent}	&\Leftrightarrow \pi (L) \in \M{s}{\Endquot{X}} \text{ nilpotent} \\
															&\Leftrightarrow \pi (L_2) \in \M{s}{\Endquot{X}} \text{ nilpotent} \\
															&\Leftrightarrow L_2 \in W \text{ nilpotent} \\
															&\Leftrightarrow \kappa \pi (L_2) \in \M{s}{\KK} \text{ nilpotent}
\end{align*}
and in conclusion
\begin{equation}
\label{eq:isotypicnilpotent}
\begin{split}
\Endnil{V} 	&= \left\lbrace L_1 + L_2 \mid L_2 \text{ nilpotent} \right\rbrace \\
			&= \left\lbrace L_1 + L_2 \mid \kappa \pi (L_2) \in \M{s}{\KK} \text{ nilpotent} \right\rbrace.
\end{split}
\end{equation}
Thus, we have to investigate nilpotent matrices in $\M{s}{\KK}$ to deepen our understanding of nilpotent endomorphisms of isotypic components. In the following we continue to speak of $\pi (L)$ for endomorphisms $L \in \End{V}$ identified with $\M{s}{\End{X}}$ and remember that this is the same as $\pi (L_2)$.

First of all, recall that matrices in $\M{s}{\KK}$ are noninvertible if they are nilpotent and furthermore that they are invertible if and only if they have full rank over $\KK$. The rank is defined to be the number of (right) linear independent column vectors or equally the number of (left) linear independent row vectors. These results are well known for real and complex matrices. For the quaternionic case consult \autocite{Zhang.1997} and the appendix. We may therefore decompose $\Endnil{V}$ as follows:
\[ \Endnil{V} = \bigcup_{i=1}^s J_i, \]
where $J_i = \left\lbrace L \in \End{V} \mid \kappa\pi(L) \text{ nilpotent,} \rk \kappa \pi(L) = s-i \right\rbrace$. Furthermore, we may embed the $J_i$ into larger collections by dropping the requirement to be nilpotent
\[ J_i \subset \Lambda_i = \left\lbrace L \in \End{V} \mid \rk \kappa \pi(L) = s-i \right\rbrace. \]
Let $\Mi{s}{\KK}$ denote the submanifold of matrices with rank $s-i$ in $\M{s}{\KK}$. Its dimension and codimension are known from \autoref{ap:codimrank} in the appendix. The $\Lambda_i$ are submanifolds of $\End{V}$ of the same codimension which can be seen from the decomposition in Equations \ref{eq:decomposition} and \ref{eq:isotypicnilpotent},
\begin{align*}
\codim \Lambda_i	&= \codim \Mi{s}{\KK} \\
					&= i^2 \dim \KK \\
					&= i^2 \Ind{X}
\end{align*}
More precisely speaking, $\pi$ and $\kappa$ are clearly surjective linear maps and therefore submersions of manifolds. Hence, $\pi^{-1} \kappa^{-1} (\Mi{s}{\KK}) \subset \End{V}$ is a submanifold of the same codimension. Recall that $i=1, \dotsc, s$ and note that this codimension is $1$ if and only if $i=1$ and $\Ind{X} = 1$ or equivalently $i=1$ and $\KK=\RR$. In that case we skip the embedding of $J_1$ into $\Lambda_1$. As is proven in \autoref{ap:realnil} in the appendix the real nilpotent matrices of rank $s-1$ form an $(s^2-s)$-dimensional submanifold of $\M{s}{\RR}$ that we call $\Mnil{s}{\RR}$. Thus the codimension of $J_1$ is
\begin{align*}
\codim J_1	&= \codim \Mnil{s}{\RR} \\
			&= s
\end{align*}
using the same argument as before for the first equality. This equals $1$ if and only if $s=1$. In this case we are considering $1\times1$ real matrices and the only nilpotent one is $0$. We summarize these results in the following theorem.
\begin{theorem}
	\label{thm:isotypicnilpotent}
	Let $X$ be an indecomposable finite-dimensional real representation of the monoid $\Sigma$ and $V=X^s$ for some $s \in \NN$. Then the set of nilpotent endomorphisms $\Endnil{V}$ is contained in the finite union of submanifolds of $\End{V}$ of codimensions:
	\begin{enumerate}[(i)]
		\item $s, i^2$ with $i=2, \dotsc, s$ if\/ $\Ind{X} = 1$ or
		\item $i^2 \Ind{X}$ with $i=1, \dotsc, s$ if\/ $\Ind{X}=2$ or $4$.
	\end{enumerate}
\end{theorem}
\begin{rem}
	This codimension equals $1$ if and only if the representation $X$ is absolutely indecomposable and the isotypic component consists of only one indecomposable summand. That is when $\Ind{X}=1$ and $s=1$.
\end{rem}

\subsection{Generalized kernels in generic one parameter familes of endomorphisms}
\label{subsec:gb}
We now return to arbitrary finite-dimensional real representations
\[ V= V_1 \oplus \dotso \oplus V_m, \]
where the $V_i$ are its isotypic components.
\begin{lemma}
	\label{lem:isotypicnilpotent}
	Let $L \in \End{V}$ be nilpotent and represented as
	\[ L = \quadmat{L}{m} \in \End{V} \]
	with $L^{ij} \in \Hom{V_j}{V_i}$. Then all the $L^{ii} \in \End{V_i}$ are nilpotent.
\end{lemma}
\begin{proof}
	Let
	\[ L = \quadmat{L}{m} \in \End{V} \]
	and $n \in \NN$ such that $L^n = 0$. The blockwise entries of $L^n$ are
	\[ \left(L^n\right)^{ij} = \sum_{1 \le l_1, \dotsc , l_{n-1} \le m} L^{il_1} L^{l_1 l_2} \dotsm L^{l_{n-1}j}. \]
	Especially for $i=j$ we have
	\begin{align*}
	0 = \left(L^n\right)^{ii}	&= \sum_{1 \le l_1, \dotsc , l_{n-1} \le m} L^{il_1} L^{l_1 l_2} \dotsm L^{l_{n-1}i} \\
								&= \left( L^{ii} \right)^n + \sum_{1 \le l_1, \dotsc , l_{n-1} \le m \atop \exists r \colon l_r \ne i} L^{il_1} L^{l_1 l_2} \dotsm L^{l_{n-1}i}.
	\end{align*}
	\autoref{lem:homisotypic} tells us that 
	\[ L^{il_1} L^{l_1 l_2} \dotsm L^{l_{n-1}i} \in \Endnilnull{V_i} \]
	whenever there exists an $r$ such that $l_r \ne i$ -- all its blockwise components are nilpotent. As $\Endnilnull{V_i}$ is an ideal the same holds for
	\[ \left( L^{ii} \right)^n = - \sum_{1 \le l_1, \dotsc , l_{n-1} \le m \atop \exists r \colon l_r \ne i} L^{il_1} L^{l_1 l_2} \dotsm L^{l_{n-1}i}. \]
	This especially yields that $\left( L^{ii} \right)^n$ is nilpotent and thus the same holds true for $L^{ii}$ which completes the proof.
\end{proof}

We have assumed that $V$ splits as a sum of isotypic components. These furthermore decompose into indecomposable representations as follows:
\[ V_i \cong X_i^{j_i}, \]
where the $X_i \subset V$ are indecomposable and $j_i \in \NN$ suitable. If $L \in \End{V}$ is an arbitrary endomorphism, its generalized kernel $\gker{L}$ is a subrepresentation with a complement -- the reduced image $\redim{L}$. Hence, by the Krull-Schmidt theorem (\ref{thm:KrullSchmidt}) it is isomorphic to the direct sum of some of the indecomposable components of $V$:
\[ \gker{L} \cong X_{i_1}^{s_1} \oplus \dotso \oplus X_{i_k}^{s_k} \]
with $k\le m$, $1 \le i_1 < \dotso < i_k \le m$ and suitable $1 \le s_r \le j_{i_r} \in \NN$ . We may therefore classify endomorphisms according to their generalized kernels. Renaming its isotypic components
\[ W_r = X_{i_r}^{s_r} \]
we denote
\[ \Iso = \left\lbrace L \in \End{V} \mid \gker{L} \cong W_1 \oplus \dotso \oplus W_k \right\rbrace. \]
\begin{theorem}
	\label{thm:codim}
	Suppose $V$ decomposes as the direct sum of indecomposables
	\[ V \cong X_1^{j_1} \oplus \dotso \oplus X_m^{j_m} \]
	where the $X_i$ are pairwise nonisomorphic. Choose $k \le m$, $1 \le i_1 < \dotso < i_k \le m$ and $1 \le s_r \le j_{i_r}$ and rename
	$W_r = X_{i_r}^{s_r}$ for $r=1,\dotsc, k$.
	Then $\Iso$ is contained in the finite union of submanifolds of codimensions
	\[ \sum_{r=1}^{k} d_r \]
	where
	\[ d_r = \begin{cases}
		 s_r , p^2 \text{ with } p=2, \dotsc , s_r 			& \text{if} \quad \Ind{X_{i_r}} = 1 \\
		 p^2 \Ind{X_{i_r}} \text{ with } p=1, \dotsc, s_r 	& \text{if} \quad \Ind{X_{i_r}} = 2 \text{ or } 4.
	\end{cases} \]
\end{theorem}
\begin{proof}
	Choose an arbitrary endomorphism $L_0 \in \Iso$ and decompose $V$ into the generalized kernel and reduced image of $L_0$:
	\[ V = \gker{L_0} \oplus \redim{L_0}. \]
	Recall from \autoref{lem:Rink63} that for an endomorphism $L \in \End{V}$ close enough to the zero endomorphism $0 \in \End{V}$ the sum $L_0+L$ is conjugate to 
	\[ \begin{pmatrix}
	\phi^{11}(L)	& 0 \\
	0				& \phi^{22}(L)
	\end{pmatrix} \]
	with respect to that decomposition. As $\phi^{22}(L) = L_0^{22} + \OO(\|L\|)$ and $L_0^{22}$ invertible, $\phi^{22}(L)$ is invertible as well. The generalized kernel $\gker{L_0+L}$ is therefore isomorphic to $\gker{L_0}$ if and only if $\phi^{11}(L)$ is nilpotent. As a matter of fact this means
	\[ L_0+L \in \Iso \Leftrightarrow \phi^{11}(L) \text{ nilpotent} \]
	for all $L$ in a suitable neighborhood $U$ of $0 \in \End{V}$.
	
	Furthermore, 
	\[ \phi^{11} \colon U \to \End{\gker{L_0}} \cong \End{W_1 \oplus \dotso \oplus W_k}\]
	with $\phi^{11}(L) = L_0^{11} + L^{11} + \OO(\|L\|^2)$. Hence $\phi^{11}$ is clearly a submersion and it suffices to prove that
	\[ \Endnil{W_1 \oplus \dotso \oplus W_k} \] 
	is contained in the union of submanifolds of $\End{W_1 \oplus \dotso \oplus W_k}$ of the specified codimensions. It then follows by an argument that has already been used before that
	\[ \left( \phi^{11} \right)^{-1} \left( \Endnil{\gker{L_0}} \right) \subset U \]
	is contained in the union of submanifolds of the same codimensions.
	
	Let $L \in \End{W_1 \oplus \dotso \oplus W_k}$ be arbitrary and decomposed respecting isotypic components
	\[ L = \quadmat{L}{k}. \]
	As we have seen in \autoref{lem:isotypicnilpotent} the block-diagonal elements are nilpotent if $L$ is nilpotent. Hence we can embed
	\[ \Endnil{W_1 \oplus \dotso \oplus W_k} \subset \left\lbrace L \in \End{W_1 \oplus \dotso \oplus W_k} \mid L^{rr} \in \Endnil{W_r} \text{ for all } 1 \le r \le k \right\rbrace, \]
	which we call $\Gamma$. \autoref{thm:isotypicnilpotent} tells us that each $\Endnil{W_r}$ is contained in the finite union of submanifolds of $\End{W_r}$ of codimensions
	\begin{equation} 
	\label{eq:dim}
	d_r = \begin{cases}
	s_r , p^2 \text{ with } p=2, \dotsc , s_r 			& \text{if} \quad \Ind{X_{i_r}} = 1 \\
	p^2 \Ind{X_{i_r}} \text{ with } p=1, \dotsc, s_r 	& \text{if} \quad \Ind{X_{i_r}} = 2 \text{ or } 4.
	\end{cases}.
	\end{equation}
	Hence $\Gamma$ is contained in the finite union of submanifolds of $\End{W_1 \oplus \dotso \oplus W_k}$ of codimensions
	\[ \sum_{r=1}^k d_r, \]
	where the $d_r$ are chosen as in \eqref{eq:dim}. This completes the proof.
\end{proof}
\begin{rem}
	Note that these codimensions are $0$ if and only if $k=0$. In that case the union of submanifolds from \autoref{thm:codim} contains all nonsingular matrices. The sum of these codimensions equals $1$ if and only if $k=1$, $s_1=1$, and $\Ind{W_1} = 1$. In that case the nilpotent endomorphisms form a real subspace and hence a proper submanifold. They are not only contained in one. In all other cases the sum of codimensions is at least $2$.
\end{rem}

We have now collected all measures to complete the proof of the main theorem.
\begin{maintheorem}
	Steady state bifurcations in one-parameter families of systems that are equivariant with respect to a finite-dimensional representation of a monoid generically occur along an absolutely indecomposable subrepresentation.
\end{maintheorem}
\begin{proof}[Proof of Main Theorem]
	Let $V$ be a finite-dimensional real representation of $\Sigma$ that decomposes as the direct sum of indecomposables
	\[ V \cong X_1^{j_1} \oplus \dotso \oplus X_m^{j_m}. \]
	 Choose $k \le m$, $1 \le i_1 < \dotso < i_k \le m$ and $1 \le s_r \le j_{i_r}$ and rename $W_r = X_{i_r}^{s_r}$ for $r=1,\dotsc, k$. \autoref{thm:codim} and the remark thereafter tell us that $\Iso$ is a submanifold of codimension $0$ if and only if $k = 0$. It is of codimension $1$ if and only if $k=1, s_1=1$, and $W_{i_1}$ is absolutely indecomposable. In all other cases $\Iso$ is contained in the finite union of submanifolds of codimension $2$ or higher. This especially means that
	\[ \zeta = \left\lbrace L \in \End{V} \mid \gker{L} \ne \lbrace 0 \rbrace \text{ is absolutely indecomposable} \right\rbrace \]
	is the finite union of submanifolds of $\End{V}$ of codimension $1$ and
	\[ \eta = \left\lbrace L \in \End{V} \mid \gker{L} \ne \lbrace 0 \rbrace \text{ is not absolutely indecomposable} \right\rbrace \]
	is contained in the finite union of submanifolds of $\End{V}$ of codimension $2$ or higher. This is due to the fact that we only have finitely many possibilities of choosing $k \le m$, \mbox{$1 \le i_1 < \dotso < i_k \le m$} and $1 \le s_r \le j_{i_r}$.
	
	Thom's transversality theorem (compare to \autocite{Hirsch.1976}) now tells us that $\zeta$ is intersected transversely -- especially in isolated points -- and $\eta$ is not intersected at all by a generic one parameter family of endomorphisms. Together with the considerations at the beginning of this section, this completes the proof.
\end{proof}

\section{An example}
\label{sec:ex}

Finally, we investigate the generic steady state bifurcations in an $8$-cell homogeneous coupled cell system to illustrate the use of the main theorem. Consider the network in \autoref{fig:ex}
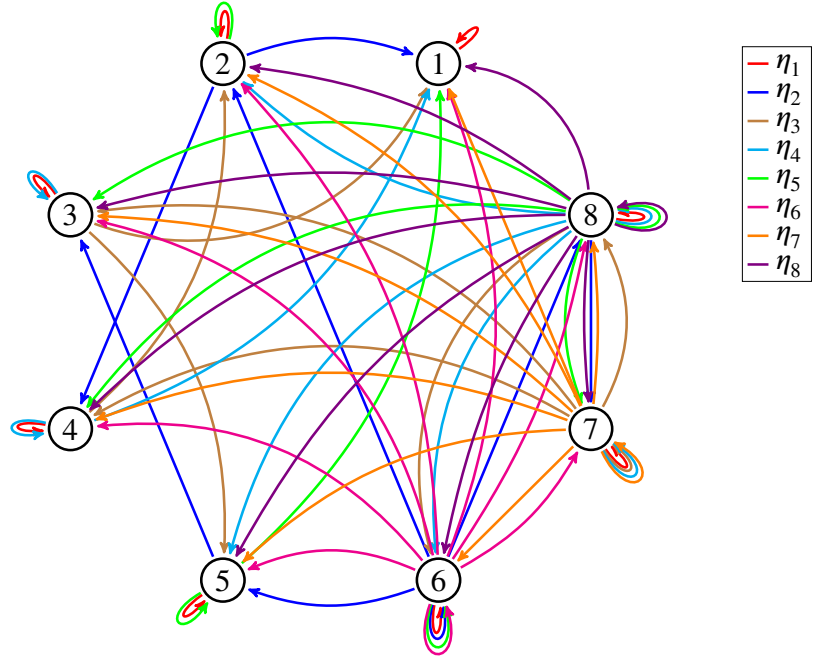
\begin{figure}[h]
	\begin{center}
		\resizebox{.7\linewidth}{!}{	
			\definecolor{c1}{named}{red}
			\definecolor{c2}{named}{blue}
			\definecolor{c3}{named}{brown}
			\definecolor{c4}{named}{cyan}
			\definecolor{c5}{named}{green}
			\definecolor{c6}{named}{magenta}
			\definecolor{c7}{named}{orange}
			\definecolor{c8}{named}{violet}
			
			\centering
			\begin{tikzpicture}[->,
			>=stealth',
			shorten >=1pt,
			auto,
			node distance=4cm,
			main node/.style={line width=2pt, circle, scale = 3,draw, font=\sffamily\scriptsize,inner sep=1.5pt}]
			\def\ngon{8}
			\node[main node, regular polygon,regular polygon sides=\ngon,minimum size=5cm,draw=none] (p) {};
			\node[main node] (1) at (p.corner 1){$1$};
			\node[main node] (2) at (p.corner 2){$2$};
			\node[main node] (3) at (p.corner 3){$3$};
			\node[main node] (4) at (p.corner 4){$4$};
			\node[main node] (5) at (p.corner 5){$5$};
			\node[main node] (6) at (p.corner 6){$6$};
			\node[main node] (7) at (p.corner 7){$7$};
			\node[main node] (8) at (p.corner 8){$8$};
			\path[every node/.style={font=\sffamily\small}, line width =2pt]
			(1) edge [in=49.21875, out=35.78125, looseness = 15, color = {c1}] node {} (1)
			(2) edge [in=94.21875, out=80.78125, looseness = 15, color = {c1}] node {} (2)
			(3) edge [in=139.21875, out=125.78125, looseness = 15, color = {c1}] node {} (3)
			(4) edge [in=184.21875, out=170.78125, looseness = 15, color = {c1}] node {} (4)
			(5) edge [in=229.21875, out=215.78125, looseness = 15, color = {c1}] node {} (5)
			(6) edge [in=274.21875, out=260.78125, looseness = 15, color = {c1}] node {} (6)
			(7) edge [in=319.21875, out=305.78125, looseness = 15, color = {c1}] node {} (7)
			(8) edge [in=364.21875, out=350.78125, looseness = 15, color = {c1}] node {} (8)
			(2) edge [bend left =20, color = {c2}] node {} (1)
			(6) edge [color = {c2}] node {} (2)
			(5) edge [color = {c2}] node {} (3)
			(2) edge [color = {c2}] node {} (4)
			(6) edge [bend left =20, color = {c2}] node {} (5)
			(6) edge [in=280.625, out=256.375, looseness = 11, color = {c2}] node {} (6)
			(8) edge [color = {c2}] node {} (7)
			(6) edge [color = {c2}] node {} (8)
			(3) edge [bend left =-42.890625, color = {c3}] node {} (1)
			(4) edge [bend left =-25.78125, color = {c3}] node {} (2)
			(7) edge [bend left =-30.671875, color = {c3}] node {} (3)
			(7) edge [bend left =-30.5625, color = {c3}] node {} (4)
			(3) edge [bend right =-25, color = {c3}] node {} (5)
			(8) edge [bend left =-38.34375, color = {c3}] node {} (6)
			(7) edge [in=324.03125, out=301.96875, looseness = 11, color = {c3}] node {} (7)
			(7) edge [bend left =-28.125, color = {c3}] node {} (8)
			(4) edge [bend left =-25.1875, color = {c4}] node {} (1)
			(8) edge [bend right =-20, color = {c4}] node {} (2)
			(3) edge [in=145.4375, out=121.5625, looseness = 11, color = {c4}] node {} (3)
			(4) edge [in=190.4375, out=166.5625, looseness = 11, color = {c4}] node {} (4)
			(8) edge [bend left =-30.9375, color = {c4}] node {} (5)
			(8) edge [bend left =-28.125, color = {c4}] node {} (6)
			(7) edge [in=327.4375, out=298.5625, looseness = 11, color = {c4}] node {} (7)
			(8) edge [in=370.4375, out=345.5625, looseness = 11, color = {c4}] node {} (8)
			(5) edge [bend left =-25.484375, color = {c5}] node {} (1)
			(2) edge [in=101.84375, out=75.15625, looseness = 11, color = {c5}] node {} (2)
			(8) edge [bend left =-34.453125, color = {c5}] node {} (3)
			(8) edge [bend left =-30.9375, color = {c5}] node {} (4)
			(5) edge [in=236.84375, out=210.15625, looseness = 11, color = {c5}] node {} (5)
			(6) edge [in=284.84375, out=253.15625, looseness = 11, color = {c5}] node {} (6)
			(8) edge [bend left =-20.390625, color = {c5}] node {} (7)
			(8) edge [in=373.84375, out=341.15625, looseness = 11, color = {c5}] node {} (8)
			(6) edge [bend left =-20.78125, color = {c6}] node {} (1)
			(6) edge [bend left =-20.5625, color = {c6}] node {} (2)
			(6) edge [bend left =-32.34375, color = {c6}] node {} (3)
			(6) edge [bend left =-28.125, color = {c6}] node {} (4)
			(6) edge [bend left =-23.90625, color = {c6}] node {} (5)
			(6) edge [in=287.25, out=250.75, looseness = 11, color = {c6}] node {} (6)
			(6) edge [bend left =-15.46875, color = {c6}] node {} (7)
			(6) edge [bend left =-11.25, color = {c6}] node {} (8)
			(7) edge [color = {c7}] node {} (1)
			(7) edge [bend left =-20.15625, color = {c7}] node {} (2)
			(7) edge [bend left =-20.234375, color = {c7}] node {} (3)
			(7) edge [bend left =-20.3125, color = {c7}] node {} (4)
			(7) edge [bend left =-20.390625, color = {c7}] node {} (5)
			(7) edge [color = {c7}] node {} (6)
			(7) edge [in=330.65625, out=295.34375, looseness = 11, color = {c7}] node {} (7)
			(7) edge [bend left =-5.625, color = {c7}] node {} (8)
			(8) edge [bend left =-39.375, color = {c8}] node {} (1)
			(8) edge [bend right = 15, color = {c8}] node {} (2)
			(8) edge [bend left =-15.125, color = {c8}] node {} (3)
			(8) edge [bend left =-22.5, color = {c8}] node {} (4)
			(8) edge [bend left =-16.875, color = {c8}] node {} (5)
			(8) edge [bend left =-11.25, color = {c8}] node {} (6)
			(8) edge [bend left =-5.625, color = {c8}] node {} (7)
			(8) edge [in=376.0625, out=338.9375, looseness = 11, color = {c8}] node {} (8)
			;
			
			\begin{customlegend}[legend cell align=left, 
			legend entries={ 
				$\eta_1$,
				$\eta_2$,
				$\eta_3$, 
				$\eta_4$,
				$\eta_5$,
				$\eta_6$,
				$\eta_7$,
				$\eta_8$
			},
			legend style={at={(13,7.5)},nodes={scale=2, transform shape}}] 
			\addlegendimage{-,c1,line width=2pt}
			\addlegendimage{-,c2,line width=2pt}
			\addlegendimage{-,c3,line width=2pt}
			\addlegendimage{-,c4,line width=2pt}
			\addlegendimage{-,c5,line width=2pt}
			\addlegendimage{-,c6,line width=2pt}
			\addlegendimage{-,c7,line width=2pt}
			\addlegendimage{-,c8,line width=2pt}
			\end{customlegend}
			
			\end{tikzpicture}
		}%
	\end{center}%
	\caption{An $8$-cell homogeneous coupled cell system.}
	\label{fig:ex}
\end{figure}%
\noindent
where each cell is subject to $1$-dimensional internal dynamics governed by its incoming arrows and a parameter $\lambda\in \RR$ via the same function $f \colon \RR^8 \times \RR \to \RR$. Denoting the state of cell $i$ by $x_i \in \RR$, it's behavior is driven by its inputs $\eta_1(i), \dotsc, \eta_8(i)$ through the ordinary differential equation
\[ \dot{x}_i = f (x_{\eta_1(i)}, \dotsc, x_{\eta_8(i)}) . \]
The corresponding parameter dependent vector field is
%
%
\begin{equation*}
F(x, \lambda) =
\begin{pmatrix}
f(\textcolor{c1}{x_{1}},\textcolor{c2}{x_{2}},\textcolor{c3}{x_{3}},\textcolor{c4}{x_{4}},\textcolor{c5}{x_{5}},\textcolor{c6}{x_{6}},\textcolor{c7}{x_{7}},\textcolor{c8}{x_{8}}, \lambda)\\
f(\textcolor{c1}{x_{2}},\textcolor{c2}{x_{6}},\textcolor{c3}{x_{4}},\textcolor{c4}{x_{8}},\textcolor{c5}{x_{2}},\textcolor{c6}{x_{6}},\textcolor{c7}{x_{7}},\textcolor{c8}{x_{8}}, \lambda)\\
f(\textcolor{c1}{x_{3}},\textcolor{c2}{x_{5}},\textcolor{c3}{x_{7}},\textcolor{c4}{x_{3}},\textcolor{c5}{x_{8}},\textcolor{c6}{x_{6}},\textcolor{c7}{x_{7}},\textcolor{c8}{x_{8}}, \lambda)\\
f(\textcolor{c1}{x_{4}},\textcolor{c2}{x_{2}},\textcolor{c3}{x_{7}},\textcolor{c4}{x_{4}},\textcolor{c5}{x_{8}},\textcolor{c6}{x_{6}},\textcolor{c7}{x_{7}},\textcolor{c8}{x_{8}}, \lambda)\\
f(\textcolor{c1}{x_{5}},\textcolor{c2}{x_{6}},\textcolor{c3}{x_{3}},\textcolor{c4}{x_{8}},\textcolor{c5}{x_{5}},\textcolor{c6}{x_{6}},\textcolor{c7}{x_{7}},\textcolor{c8}{x_{8}}, \lambda)\\
f(\textcolor{c1}{x_{6}},\textcolor{c2}{x_{6}},\textcolor{c3}{x_{8}},\textcolor{c4}{x_{8}},\textcolor{c5}{x_{6}},\textcolor{c6}{x_{6}},\textcolor{c7}{x_{7}},\textcolor{c8}{x_{8}}, \lambda)\\
f(\textcolor{c1}{x_{7}},\textcolor{c2}{x_{8}},\textcolor{c3}{x_{7}},\textcolor{c4}{x_{7}},\textcolor{c5}{x_{8}},\textcolor{c6}{x_{6}},\textcolor{c7}{x_{7}},\textcolor{c8}{x_{8}}, \lambda)\\
f(\textcolor{c1}{x_{8}},\textcolor{c2}{x_{6}},\textcolor{c3}{x_{7}},\textcolor{c4}{x_{8}},\textcolor{c5}{x_{8}},\textcolor{c6}{x_{6}},\textcolor{c7}{x_{7}},\textcolor{c8}{x_{8}}, \lambda)\\
\end{pmatrix}.
\end{equation*}
We investigate generic steady state bifurcations in this network using the method introduced in \autocite{Rink.2014} (in that setting the network is its own fundamental network). The network vector fields are precisely those vector fields on $\RR^8$ that are equivariant with respect to the representation of the monoid $\Sigma$ with eight elements generated by the linear transformations
\begin{align*}
\sigma_1 \colon (x_1,x_2,x_3,x_4,x_5,x_6,x_7,x_8) &\mapsto (x_2,x_6,x_4,x_8,x_2,x_6,x_7,x_8), \\
\sigma_2 \colon (x_1,x_2,x_3,x_4,x_5,x_6,x_7,x_8) &\mapsto (x_3,x_5,x_7,x_3,x_8,x_6,x_7,x_8).
\end{align*}
To translate the equivariance to the parameter dependent systems we assume the action is only on the spatial variable. 

To investigate steady state bifurcations we assume the existence of the trivial branch of solutions
\[ F(0,\lambda) = 0 \]
for all parameter values $\lambda \in \RR$. Furthermore, we want the bifurcation to occur at $\lambda_0=0$ which can only happen if $D_x F(0,0)$ is noninvertible due to the implicit function theorem. This means that the linearization $D_x F(0,0)$ has a nontrivial generalized kernel along which the steady state bifurcations may occur (compare to the beginning of \autoref{sec:gb}).

The representation of $\Sigma$ decomposes into four absolutely indecomposable components
\[ \RR^8 = X \oplus Y \oplus V \oplus W \]
where
\begin{align*}
X	&= \left\lbrace x_1 = \dotso = x_8 \right\rbrace, \\
Y	&= \left\lbrace x_2 = \dotso = x_8 = 0 \right\rbrace, \\
V	&= \left\lbrace x_1 = x_4, x_2 = x_5 = x_6 = x_8 = 0 \right\rbrace, \\
W	&= \left\lbrace x_1 = x_5, x_3 = x_4 = x_7 = x_8 = 0 \right\rbrace.
\end{align*}
The main theorem tells us that, generically, branches of steady states bifurcate off the trivial solution along one of these components meaning that the generalized kernels of linearizations of vector fields are generically equivalent to one of the components as subrepresentations. In each case we perform the equiavariant Lyapunov-Schmidt reduction (see \autocite{Rink.2014}) to restrict to an equivariant equation on the subrepresentation.

The subrepresentations $X$ and $Y$ are both one-dimensional and both transformations act trivially on them. On $X$ both $\sigma_1$ and $\sigma_2$ act as identity whereas $\sigma_1$ and $\sigma_2$ both act as zero on $Y$. Therefore, we expect a transcritical bifurcation in both cases. The transcritical bifurcation on $X$ is fully synchronous. The one on $Y$ occurs only in cell $1$. This is due to the fact that cell $1$ has no outgoing arrows into any other cell.

Choosing the basis
\[ (1,0,0,1,0,0,0,0)^T, (0,0,1,0,0,0,0,0)^T, (0,0,0,0,0,0,1,0)^T \]
for $V$ and
\[ (1,0,0,0,1,0,0,0)^T, (0,1,0,0,0,0,0,0)^T, (0,0,0,0,0,1,0,0)^T \]
for $W$ the transformations act as matrices $a, a'$ for $\sigma_1$ and $b, b'$ for $\sigma_2$ on $V$ and $W$, respectively, where
\[ a = \begin{pmatrix}
0	& 0	& 0 \\
1	& 0	& 0 \\
0	& 0	& 1
\end{pmatrix}, \quad
a' = \begin{pmatrix}
0	& 1	& 0 \\
0	& 0	& 1 \\
0	& 0	& 1
\end{pmatrix}, \quad
b = \begin{pmatrix}
0	& 1	& 0 \\
0	& 0	& 1 \\
0	& 0	& 1
\end{pmatrix}, \quad
b' = \begin{pmatrix}
0	& 0	& 0 \\
1	& 0	& 0 \\
0	& 0	& 1
\end{pmatrix}. \]
Therefore, $V$ and $W$ are equivalent as subrepresentations via the isomorphism $\varphi$ with matrix representation
\[ \begin{pmatrix}
0	& -1	& 1 \\
-1	& 0		& 1 \\
0	& 0		& 1
\end{pmatrix} \]
for $\varphi$ and $\varphi^{-1}$. Hence, we may restrict to $V$ as the generic steady state bifurcations on $W$ are the same in their specific coordinates. They only differ by their respective choice of a basis.

Performing the equivariant Lyapunov-Schmidt reduction onto $V$ and choosing coordinates $v_1,v_2,v_3$ we obtain the bifurcation equation
\[ r(v,\lambda) = \begin{pmatrix}
r_1 (v_1,v_2,v_3, \lambda) \\ r_2 (v_1,v_2,v_3, \lambda) \\ r_3 (v_1,v_2,v_3, \lambda)
\end{pmatrix} = 0 \]
with the properties inherited from $F$,
\begin{align*}
r (0,\lambda) = 0 \quad \text{for all} \quad \lambda \in \RR, \\
D_v r(0,0) \quad \text{has eigenvalue } 0.
\end{align*}
Furthermore, $r$ is equivariant in its spatial component with respect to $\sigma_1$ and $\sigma_2$ (which act as $a$ and $b$ on $V$). This provides the additional properties
\begin{alignat*}{4}
& r_1(0,v_1,v_3,\lambda)	&& = 0, \qquad							&& r_1(v_2,v_3,v_3,\lambda)	&= r_2(v_1,v_2,v_3,\lambda), \\
& r_2(0,v_1,v_3,\lambda)	&& = r_1 (v_1,v_2,v_3,\lambda), \qquad	&& r_2(v_2,v_3,v_3,\lambda)	&= r_3(v_1,v_2,v_3,\lambda), \\
& r_3(0,v_1,v_3,\lambda)	&& = r_3 (v_1,v_2,v_3,\lambda), \qquad	&& r_3(v_2,v_3,v_3,\lambda)	&= r_3(v_1,v_2,v_3,\lambda).
\end{alignat*}
These restrictions yield that up to second order we have to solve the equations
\begin{align*}
\alpha \lambda v_1 + \beta v_1^2 + \gamma v_1v_3 + \OO ( |\lambda|^2, \|(v,\lambda)\|^3) &= 0, \\
\alpha \lambda v_2 + \beta v_2^2 + \gamma v_2v_3 + \OO ( |\lambda|^2, \|(v,\lambda)\|^3) &= 0, \\
\alpha \lambda v_3 + (\beta + \gamma ) v_3^2 + \OO ( |\lambda|^2, \|(v,\lambda)\|^3) &= 0.
\end{align*}
Under the generic conditions that $\alpha, \beta \ne 0$ and $\gamma \ne - \beta$, we are given eight branches of solutions
\begin{alignat*}{6}
& v_1	&&= 0 \qquad														&& v_2	&&= 0 \qquad														&& v_3	&&= 0, \\
& v_1	&&= -\frac{\alpha}{\beta} \lambda + \OO(\lambda^2) \qquad			&& v_2	&&= 0 \qquad														&& v_3	&&= 0, \\
& v_1	&&= 0 \qquad														&& v_2	&&= -\frac{\alpha}{\beta} \lambda + \OO(\lambda^2) \qquad			&& v_3	&&= 0, \\
& v_1	&&= -\frac{\alpha}{\beta} \lambda + \OO(\lambda^2) \qquad			&& v_2	&&= -\frac{\alpha}{\beta} \lambda + \OO(\lambda^2) \qquad			&& v_3	&&= 0, \\
& v_1	&&= 0 \qquad														&& v_2	&&= 0 \qquad														&& v_3	&&= -\frac{\alpha}{\beta+\gamma} \lambda + \OO(\lambda^2) ,\\
& v_1	&&= -\frac{\alpha}{\beta+\gamma} \lambda + \OO(\lambda^2) \qquad	&& v_2	&&= 0 \qquad														&& v_3	&&= -\frac{\alpha}{\beta+\gamma} \lambda + \OO(\lambda^2) ,\\
& v_1	&&= 0 \qquad														&& v_2	&&= -\frac{\alpha}{\beta+\gamma} \lambda + \OO(\lambda^2) \qquad	&& v_3	&&= -\frac{\alpha}{\beta+\gamma} \lambda + \OO(\lambda^2) ,\\
& v_1	&&= -\frac{\alpha}{\beta+\gamma} \lambda + \OO(\lambda^2) \qquad	&& v_2	&&= -\frac{\alpha}{\beta+\gamma} \lambda + \OO(\lambda^2) \qquad	&& v_3	&&= -\frac{\alpha}{\beta+\gamma} \lambda + \OO(\lambda^2) .
\end{alignat*}
Returning to the original system these imply the coexistence of eight solution branches (the trivial one and seven transcritical ones) with different cells being synchronous:
\begin{align*}
x_1	&= x_2 = x_3 = x_4 = x_5 = x_6 = x_7 = x_8 = 0 , \\
x_1	&= x_4 = -\frac{\alpha}{\beta} \lambda + \OO(\lambda^2), \quad x_2 = x_3 = x_5 = x_6 = x_7 = x_8 = 0, \\
x_1 &= x_2 = x_4 = x_5 = x_6 = x_7 = x_8 = 0, \quad x_3 = -\frac{\alpha}{\beta} \lambda + \OO(\lambda^2), \\
x_1 &= x_3 = x_4 = -\frac{\alpha}{\beta} \lambda + \OO(\lambda^2), \quad x_2 = x_5 = x_6 = x_7 = x_8 = 0, \\
x_1	&= x_2 = x_3 = x_4 = x_5 = x_6 = x_8 = 0, \quad x_7 = -\frac{\alpha}{\beta+\gamma} \lambda + \OO(\lambda^2), \\
x_1	&= x_4 = x_7 = -\frac{\alpha}{\beta+\gamma} \lambda + \OO(\lambda^2), \quad x_2 = x_3 = x_5 = x_6 = x_8 = 0, \\
x_1	&= x_2 = x_4 = x_5 = x_6 = x_8 = 0, \quad x_3 = x_7 = -\frac{\alpha}{\beta+\gamma} \lambda + \OO(\lambda^2), \\
x_1	&= x_3 = x_4 = x_7 = -\frac{\alpha}{\beta+\gamma} \lambda + \OO(\lambda^2), \quad x_2 = x_5 = x_6 = x_8 = 0.
\end{align*}

\begin{appendix}
\renewcommand\thetheorem{\thesubsection.\arabic{theorem}}
\makeatletter
\@addtoreset{theorem}{subsection}
\makeatother
\renewcommand\theequation{\thesubsection.\arabic{equation}}
\makeatletter
\@addtoreset{equation}{subsection}
\makeatother
\section{Appendix}
We want to prove some results on submanifolds of the space $\M{n}{\KK}$ of $n\times n$ matrices over \mbox{$\KK = \RR, \CC$,} or $\HH$. We treat $\M{n}{\KK}$ as a real vector space meaning that we restrict scalar multiplication to real numbers. The result in the first proposition is well-known in the cases $\KK = \RR$ or $\KK = \CC$. The necessary arguments for the real case are sketched as an exercise in the equally known book \autocite{Guillemin.1974}. We state them here especially for the case $\KK = \HH$ even though they hold true for all three cases. For more details on quaternionic matrices consult \autocite{Zhang.1997}. We summarize some facts and definitions that are especially important for our considerations. 

Let us consider scalar multiplication by quaternions for a moment. As multiplication of quaternions is not commutative, we have to distinguish between left and right scalar multiplication and \emph{left} and \emph{right linear (in)dependence} (over $\HH$). The \emph{rank} of a quaternionic matrix $A$ is the number of right linear independent column vectors or equally the number of left linear independent row vectors of $A$. A quadratic matrix $A \in \M{n}{\HH}$ is invertible (there exists $B\in \M{n}{\HH}$ such that $AB = BA = \mathbbm{1}$) if and only if it has full rank $n$. Furthermore, $\rk PAQ = \rk A$ for any invertible matrices $P$ and $Q$ of suitable dimensions.
\begin{prop}
	\label{ap:codimrank}
	Let\/ $\KK = \RR, \CC$, or $\HH$ and $\M{n}{\KK}$ be the space of all $n \times n$ matrices with entries in $\KK$ considered as a real vector space. 
	Then
	\[ \Mr{n}{\KK} = \left\lbrace A \in \M{n}{\KK} \mid \rk A = r \right\rbrace \]
	with $r=0, \dotsc, n$ is a submanifold of codimension $(n-r)^2 \dim \KK$.
\end{prop}
\begin{proof}
	As $\mathbf{M}_0 \left(n ; \KK \right) = \lbrace 0 \rbrace$ and $\mathbf{M}_n \left(n ; \KK \right) = \lbrace A \in \M{n}{\KK} \colon A \text{ invertible} \rbrace$, the special cases $r=0$ and $r=n$ are clear. Hence, let $r \in \lbrace 1 , \dotsc , n-1 \rbrace$ and $L \in \Mr{n}{\KK}$. Then $L$ has  $r$ (right) linear independent column vectors $v_1, \dotsc , v_r$. Without loss of generality (by exchanging columns of $L$), we may assume that these are the first $r$ columns of $L$. The $n\times r$ matrix $(v_1, \dotsc, v_r)$ consisting of those column vectors still has rank $r$. Hence it has $r$ (left) linear independent row vectors. The exchanging of rows of said matrix allows us to assume that the first $r$ rows are (left) linear independent. Applying the same exchange of rows to the full matrix $L$ allows us to assume 
	\[ L = \begin{pmatrix}
	A	& B \\
	C	& D
	\end{pmatrix}, \]
	where $A \in \M{r}{\KK}$ is invertible -- especially $\rk A = r$ -- and $B \in \M{r \times (n-r)}{\KK}$, \mbox{$C \in \M{(n-r) \times r}{\KK}$} and $D \in \M{(n-r)}{\KK}$. Consider the matrix
	\[ L_0 = \begin{pmatrix}
	\mathbbm{1}	& -A^{-1}B \\
	0			& \mathbbm{1}
	\end{pmatrix}, \]
	where the dimensions of the identity matrices are suitably chosen. Then 
	\[ LL_0 = \begin{pmatrix}
	A	& 0 \\
	C	& -CA^{-1}B+D
	\end{pmatrix}. \]
	As $L_0$ is clearly invertible
	\[ r = \rk L = \rk LL_0 = \rk A + \rk \left(-CA^{-1}B+D\right) . \]
	But already $\rk A = r$ and thus
	\[ \rk \left(-CA^{-1}B+D\right) = 0, \]
	which is only fulfilled by the zero matrix $0 \in \M{(n-r)}{\KK}$. Furthermore, these considerations show that any block matrix 
	\[ K= \begin{pmatrix}
	\alpha	& \beta \\
	\gamma	& \delta
	\end{pmatrix}, \]
	whose upper left block $\alpha \in \M{r}{\KK}$ is invertible, has rank $r$ if and only if
	\[ -\gamma \alpha^{-1} \beta +\delta = 0. \]
	
	We may now choose a suitably small neighborhood $U$ around $L$ in a suitable topology such that every $K \in U$ is of the form
	\[ K= \begin{pmatrix}
	\alpha	& \beta \\
	\gamma	& \delta
	\end{pmatrix} \]
	with $\alpha \in \M{r}{\KK}$ invertible. If $K$ is close to $L$, then $\alpha$ is close to $A$ and hence is invertible as well. To see this in the quaternionic case we may use Theorem 7.3 in \autocite{Zhang.1997} which dates back to Wolf in 1936 and connects the rank of a quaternionic matrix to that of its complex adjoint matrix. On the neighborhood $U$ we define a map
	\begin{align*}
	f \colon U 	&\to \M{(n-r)}{\KK}, \\
	K			&\mapsto -\gamma \alpha^{-1} \beta +\delta.
	\end{align*}
	This map is smooth and we have seen that $K \in U$ has rank $r$ if and only if $f(K)=0$.
	
	We have to check that the derivative of $f$ at $K=L$ is surjective on tangent spaces. As the target space of $f$ is a linear space, its tangent space is the same space $\M{(n-r)}{\KK}$. To prove surjectivity let $X \in \M{(n-r)}{\KK}$ be arbitrary and consider the smooth curve
	\[ \nu (t) = L + t \begin{pmatrix}
	0	& 0 \\
	0	& X
	\end{pmatrix} \]
	with $t$ being restricted to an interval around $0$ so that $\nu (t) \in U$ for all $t$. Then
	\[ f(\nu(t)) = -\gamma \alpha^{-1} \beta +\delta + tX \]
	and
	\[ \frac{\mathrm{d}}{\mathrm{d}t} f(\nu(0)) = X. \]
	This proves surjectivity of $\mathrm{D}f$. Hence $\Mr{n}{\KK} \cap U = f^{-1}(0)$ is a submanifold of codimension
	\[ \codim \Mr{n}{\KK} = \dim \M{(n-r)}{\KK} = (n-r)^2 \dim \KK, \]
	which completes the proof.
\end{proof}

The next proposition treats nilpotent real matrices whose rank is reduced by one. Embedding them into all matrices of that rank would only yield a codimension $1$ submanifold using the last proposition. The proof relies on the normal form of matrices presented in \autocite{Arnold.1971}.
\begin{prop}
	\label{ap:realnil}
	The collection of real nilpotent matrices of rank $n-1$ is a submanifold of\/ $\M{n}{\RR}$ of codimension $n$.
\end{prop}
\begin{proof}
	Let $L \in \M{n}{\RR}$ be nilpotent and $\rk L = n-1$. This directly yields that the Jordan normal form of $L$ consists of precisely one Jordan block
	\[ \begin{pmatrix}
	0	& 1			& 			& \\
		& \ddots	& \ddots	& \\
		&			& \ddots	& 1 \\
		&			&			& 0
	\end{pmatrix}. \]
	\autocite{Arnold.1971} presents a matrix normal form that depends smoothly on the matrix -- more precisely speaking a versal deformation. Any suitably small perturbation $K$ of $L$ is conjugate to a matrix
\begin{equation}
	\label{eq:normal form}
	\tag{\ensuremath{\star}}
	\begin{pmatrix}
		0	& 1			& 			& \\
			& \ddots	& \ddots	& \\
			&			& 0			& 1 \\
		a_1	& \dots		& \dots		& a_n
	\end{pmatrix}
\end{equation}
	with $a_1, \dotsc, a_n \in \RR$ depending smoothly on $K$. As $K$ is an $n\times n$ matrix it is nilpotent if and only if $K^n = 0$. However, $K^n$ is conjugate to
	\[ \begin{pmatrix}
	a_1	& \dots	& a_n \\
		&*		&
	\end{pmatrix}, \]
	which can only be $0$ if $a_i = 0$ for all $i=1,\dotsc, n$. This yields that $K$ close to $L$ is nilpotent (and of rank $n-1$) if and only if it is conjugate to $L$. The deformation in \eqref{eq:normal form} is constructed to have the minimal number of parameters which is $n$. It equals the codimension of the conjugacy orbit of $L$. Therefore, the collection of matrices conjugate to $L$ is a submanifold of $\M{n}{\RR}$ of codimension $n$.
\end{proof}

\end{appendix}

\renewcommand{\abstractname}{Acknowledgements}
\begin{abstract}
\noindent
The author would like to thank Reiner Lauterbach for countless discussions and remarks that were of great help in setting up the proof in its current form.
\end{abstract}

\begingroup
\RaggedRight
\printbibliography
\endgroup
\phantomsection
\addcontentsline{toc}{section}{References} 
\end{document}